\documentclass[12pt]{article}

\usepackage{amsmath,amsthm}
\usepackage{mathtools}
\usepackage{xcolor}
\usepackage{amssymb}
\usepackage{amsfonts}
\usepackage{graphicx}
\usepackage{booktabs}
\usepackage{units}
\usepackage{xcolor}
\usepackage{comment}
\usepackage{caption}
\usepackage{subcaption} 
\usepackage{dblfloatfix}   
\usepackage{algorithm}
\usepackage{algorithmicx}
\usepackage{algpseudocode}
\algrenewcommand\algorithmicindent{0.5em}
\usepackage{commath}

\usepackage{hyperref}
\hypersetup{
    colorlinks=true,
    linkcolor=blue,     
    urlcolor=blue,
}
\usepackage{xparse}

\usepackage{tikz}
\usepackage{pgfplots}
\usepackage{appendix}

\usepackage{enumitem}

\floatname{algorithm}{Pseudo code}
\DeclareMathOperator{\ran}{ran}
\DeclareMathOperator{\sym}{sym}

\newtheorem{definition}{Definition}

\allowdisplaybreaks
\newcounter{thmcounter}[section]  
\renewcommand{\thethmcounter}{\thesection.\arabic{thmcounter}}

\newcommand{\defthmwithqed}[2]{%
  \NewDocumentEnvironment{#1}{ o }{%
    \refstepcounter{thmcounter}
    \begin{trivlist}%
      \item[\hskip \labelsep \bfseries #2~\thethmcounter%
        \IfValueT{##1}{\ (##1)}.]%
  }{%
    \hfill $\square$%
    \end{trivlist}%
  }%
}

\newtheorem{theorem}{Theorem}
\newtheorem{lemma}{Lemma}

\newtheorem{proposition}{Proposition}

\newtheorem{corollary}{Corollary}
\newtheorem{remark}{Remark}

\usepackage{tikz}
\usetikzlibrary{arrows.meta,positioning,calc,decorations.pathmorphing}

{\left(\begin{smallmatrix}}
{\end{smallmatrix}\right)}

{\left[\begin{smallmatrix}}
{\end{smallmatrix}\right]}%

\title{Dissipativity properties of a class of nonlinear time-delay systems via Bessel--Legendre inequalities}
\author{Ikram El Haskouki\footnote{AIMCE Laboratory, ENSAM, Hassan II University of Casablanca, 150 Nile Boulevard, Casablanca 20670, Morocco and University of Wuppertal, Gaußstraße 20, 42119 Wuppertal, Germany (e-mail: \texttt{ikram.elhaskouki[at]gmail.com})} and Hannes Gernandt\footnote{University of Wuppertal, Gaußstraße 20, 42119 Wuppertal and Fraunhofer IEG, Fraunhofer Research Institution for Energy Infrastructures and Geotechnologies IEG, 03046 Cottbus, Germany  (e-mail: \texttt{gernandt[at]uni-wuppertal.de})} 
}

\date{\today}

\begin{document}

\maketitle

\begin{abstract}
Time delays are inherent in many physical and engineered systems and can significantly affect 
their stability and performance. In this work, we investigate the dissipativity of a class of nonlinear time-delay systems with multiple discrete delays and derive sufficient conditions for both delay-dependent and delay-independent dissipativity using Bessel--Legendre inequalities. For linear systems, the resulting  dissipativity conditions are expressed in terms of linear matrix inequalities (LMIs) which can be solved numerically to obtain Lyapunov--Krasovskii-type storage functions. 
\end{abstract}

\maketitle
                  
\section{Introduction}
Time delays arise naturally in many physical, biological, and engineering systems, for instance due to transport phenomena, communication networks, measurement processes, or delayed feedback. Their presence may significantly influence both stability and performance and can lead to qualitatively different behavior compared with delay-free systems. For this reason, the analysis and control of time-delay systems has been an active research area for several decades; see, for instance, \cite{Fridman2014,Scholl2024} and the references therein.

Dissipativity provides a powerful framework for the analysis and design of dynamical systems through input--output energy balance relations. It was originally introduced by Willems \cite{willems1972dissipative} and further developed by Hill and Moylan \cite{hill1976stability}. In this framework, the change of a suitable storage function is bounded by a prescribed supply rate, which describes the energy exchanged between the system and its environment. Important system properties such as impedance passivity and scattering passivity can be formulated as special cases of dissipativity by choosing appropriate quadratic supply rates. This makes dissipativity particularly useful for stability analysis, robust control design, and energy-based modeling frameworks, and in particular in form of port-Hamiltonian systems \cite{Brogliato2007,BreitHU24,haskouki2026delaydependentpassivitystabilitylinear,Hinsen2026,bustos2026classporthamiltoniansystemsgeneral,Sch17}.

In the presence of time delays, dissipativity analysis becomes more involved, since the system state depends not only on its current value but also on its past history. Consequently, classical Lyapunov functions are generally not sufficient to capture the influence of the delayed state. Instead, Lyapunov--Krasovskii storage functionals are commonly employed, as they explicitly incorporate information on the state history \cite{Fridman2014}. A central difficulty in this approach is the estimation of integral terms involving delayed states and their derivatives. Jensen's inequality has been widely used for this purpose, but it may lead to conservative conditions. To reduce this conservatism, several refined integral inequalities have been proposed, including Wirtinger-based inequalities \cite{seuret2013jensen,seuret2013wirtinger}.

Recently, Bessel--Legendre inequalities have emerged as a powerful extension of Jensen and Wirtinger-type estimates \cite{seuret2015hierarchy}. By exploiting orthogonal projections onto Legendre polynomial bases, these inequalities provide a hierarchy of increasingly accurate estimates for integral terms while preserving a tractable structure. They are therefore well suited for deriving less conservative Lyapunov--Krasovskii conditions for time-delay systems.

Motivated by these developments, this paper investigates dissipativity properties of a class of nonlinear time-delay systems with multiple constant delays. We focus on general quadratic supply rates, which include impedance and scattering passivity as important special cases. The proposed approach combines Lyapunov--Krasovskii storage functionals with Bessel--Legendre inequalities in order to derive sufficient delay-dependent dissipativity conditions. A delay-independent criterion is obtained as a special case by considering a simplified storage functional. Moreover, for linear time-delay systems, the resulting conditions are expressed as linear matrix inequalities, which enables efficient numerical verification. Recent results on delay-dependent stability and passivity of linear port-Hamiltonian delay systems \cite{haskouki2026delaydependentpassivitystabilitylinear} further motivate the study of such dissipativity conditions in the presence of delays.

The main contributions of this paper are summarized as follows.
\begin{itemize}
\item We derive sufficient delay-dependent dissipativity conditions for nonlinear systems with multiple constant delays by means of Lyapunov--Krasovskii storage functionals and Bessel--Legendre inequalities.
\item We obtain delay-independent dissipativity conditions as a special case based on a simplified storage functional.
\item We specialize the proposed framework to linear time-delay systems and derive tractable LMI conditions for dissipativity with respect to general quadratic supply rates.
\end{itemize}

The paper is organized as follows. After introducing the notation, Section~\ref{sec:dissipative} recalls the required properties of Legendre polynomials and the Bessel--Legendre inequality. Then we  establish delay-dependent and delay-\\independent dissipativity conditions for nonlinear time-delay systems. Section~\ref{sec:linear} treats the linear case and derives corresponding LMI conditions. In Section~\ref{sec:examples} we present two numerical examples where the Bessel--Legendre inequalities help to establish dissipativity. Finally, the paper concludes with a summary and an outlook on future research directions in Section~\ref{sec:conclusion}.

\section*{Notation}
Throughout, we consider multiple delays $\tau_q\ge\cdots\ge\tau_1>0$, i.e.\  $\tau_q>0$ denotes the largest delay. Let $T \in (0,\infty]$, then for a continuous function
$x:[-\tau_q,T)\to\mathbb R^n,$ and $t\in[0,T)$, we denote its history
by $x_t(\theta)=x(t+\theta)$, $\theta\in[-\tau_q,0]$.
The Sobolev space $H^1([-\tau_q,0],\mathbb R^n)$ is the space of all weakly differentiable functions in the Lebesgue space $L^2([-\tau_q,0],\mathbb R^n)$ whose weak derivative is also in that space. The Euclidean norm is denoted by $\|\cdot\|$.
For a real matrix $P$, the inequalities $P\succ0$ and $P\succeq0$
mean that $P$ is symmetric positive definite and positive semidefinite,
respectively. We write $\sym(M)=(M+M^\top)/2$ for the symmetric
part of a square matrix $M$. 
For matrices $X_i\in\mathbb{R}^{n_i\times m}$, $i=1,\ldots,k$, we denote the matrix that is obtained by stacking the columns by $\operatorname{col}(X_1, \ldots, X_k)\in\mathbb{R}^{(n_1+\ldots+n_k)\times m}$.
\section{Dissipativity of nonlinear time-delay systems}
\label{sec:dissipative}
In this section, we consider a class of nonlinear time-delay systems that are given by 
\begin{align}
\nonumber 
\dot x(t) &= f(x(t),x(t-\tau_1),\ldots, x(t-\tau_q)) 
  + g(x(t))u(t),\quad t\geq 0,\\
x(t) &= \phi(t),\quad  t \in [-\tau_q , 0], \label{eq:cnt_time_delay_system_nonlinear}
\\
y(t) &= h(x(t)),\quad t \geq  0,\nonumber 
\end{align}
Here $x$, $u$, and $y$ denote the state, input, and output, with values in
$\mathbb R^n$, $\mathbb R^m$, and $\mathbb R^m$, respectively. The maps
$f:\mathbb R^{n(q+1)}\to\mathbb R^n$ and
$g:\mathbb R^n\to\mathbb R^{n\times m}$ are locally Lipschitz
continuous, $h:\mathbb R^n\to\mathbb R^m$ is continuous, and
$\tau_q\ge\cdots\ge\tau_1>0$ for an integer $q\ge1$. We consider
histories $\phi\in H^1([-\tau_q,0],\mathbb R^n)$ and inputs
$u\in L^2_{\mathrm{loc}}([0,\infty),\mathbb R^m)$.
The system admits a unique maximal Carath\'eodory solution on
$[-\tau_q,T_{\max})$ for some $T_{\max}\in(0,\infty]$.
\begin{definition}
The delay system~\eqref{eq:cnt_time_delay_system_nonlinear} is called
\emph{dissipative} with respect to a continuous supply rate
$s:\mathbb R^m\times\mathbb R^m\to\mathbb R$ if there exists a
storage functional
$\mathcal S:H^1([-\tau_q,0],\mathbb R^n)\to[0,\infty)$ with
$\mathcal S(0)=0$ such that, for every admissible initial history and
input and every $0\le t_0\le t_1<T_{\max}$,
\begin{equation}\label{eq:dissip_ineq_cont_time}
\mathcal S(x_{t_1})-\mathcal S(x_{t_0})
\le\int_{t_0}^{t_1}s(y(t),u(t))\,\mathrm dt.
\end{equation}
The system is \emph{dissipative independently of the delays} if it is
dissipative with respect to the same supply rate for every ordered
choice $\tau_q\ge\cdots\ge\tau_1>0$.
\end{definition}
In this paper, we focus on quadratic supply rates of the form 
\[s(y,u) =\label{eq:supply_rate}
\begin{bmatrix} y\\ u \end{bmatrix}^\top
\begin{bmatrix} S_{11} & S_{12} \\ S_{12}^\top & S_{22} \end{bmatrix}
\begin{bmatrix} y\\ u \end{bmatrix},\]
\[S_{11}=S_{11}^\top\in\mathbb{R}^{m \times m}, \quad S_{12}\in\mathbb{R}^{m\times m},\quad S_{22}=S_{22}^\top \in \mathbb{R}^{m \times m}.
\]
\vspace{1mm}
\\
Systems which are dissipative with respect to the supply rate 
\[
s_{imp}(y,u)=\frac12\begin{bmatrix}
    y\\ u
\end{bmatrix}^\top\begin{bmatrix}
    0&I_m\\ I_m&0
\end{bmatrix}\begin{bmatrix}
    y\\ u
\end{bmatrix},\quad s_{sca}(y,u)=\frac12\begin{bmatrix}
    y\\ u
\end{bmatrix}^\top\begin{bmatrix}
    -I_m &0\\ 0&I_m
\end{bmatrix}\begin{bmatrix}
    y\\ u
\end{bmatrix}
\]
are called \emph{impedance} and \emph{scattering passive} respectively.
\subsection{Legendre Polynomials and the Bessel--Legendre inequality}
In the following, we provide a brief overview of the Legendre polynomials and their relevant properties.
\begin{definition}
The \emph{Legendre polynomials} considered over the interval $[-\tau_i, 0]$ are defined by: 
\begin{equation}
\forall k \in \mathbb{N}_0,\,i=1,\ldots,q: \quad L_k^i(s) = (-1)^k \sum_{l=0}^k (-1)^l \binom{k}{l} \binom{k+l}{l} \left( \frac{s+\tau_i}{\tau_i} \right)^l.
\end{equation}
\end{definition}

The Legendre polynomials have the following properties, see e.g. \cite{seuret2015hierarchy}.

\begin{lemma}\rm\ \label{Lemma:Legendre_polynomials_properties}
Consider the Legendre polynomials $L_k^i$ for $i=1,\ldots,q$ and $k \in \mathbb{N}_0$. Then the following holds 
\begin{itemize}
\item \textit{Orthogonality:} For all $(k, l) \in \mathbb{N}_0^2$ and all $i=1,\ldots,q$,
\begin{equation}\nonumber
    \int_{-\tau_i}^0 L_k^i(s)L_l^i(s) \, \mathrm{d}s = 
    \begin{cases} 
        0, & k \neq l, \\
        \frac{\tau_i}{2k+1}, & k = l.
    \end{cases}
\end{equation}
\item \textit{Boundary conditions:}
$    \forall k \in \mathbb{N}_0, \quad L_k^i(0) = 1, \quad L_k^i(-\tau_i) = (-1)^k$,\\ $i=1,\ldots,q$.

\item \textit{Differentiation:} The derivative $\dot{L}^i_k$ is given by:
\begin{equation}\nonumber
    \dot{L}^i_k(s) = 
    \begin{cases} 
        0, & k = 0, \\
        \displaystyle \sum_{j=0}^{k-1} \frac{2j+1}{\tau_i} \left(1 - (-1)^{k+j}\right) L^i_j(s), & k \ge 1.
    \end{cases}
\end{equation}
\end{itemize}
\end{lemma}
The following form of the Bessel--Legendre inequality is based on
\cite[Corollary~4]{seuret2015hierarchy}.
\begin{lemma}\label{lemma:legendre}
Let $i\in\{1,\ldots,q\}$, $\tau_i>0$,
$R_i\succeq0$, and
$x_t\in H^1([-\tau_i,0],\mathbb R^n)$. Then, for every
$N\in\mathbb N_0$,
\begin{equation}\label{eq:bessel}
\int_{-\tau_i}^0 \dot{x}_t(s)^\top R_i \dot{x}_t(s)\,\mathrm ds
\ge\frac1{\tau_i}\xi_N^i(t)^\top
\left[\sum_{k=0}^N(2k+1)\Gamma_N(k)^\top R_i\Gamma_N(k)\right]
\xi_N^i(t),
\end{equation}
where $\Omega_k^i(t)=\int_{-\tau_i}^0L_k^i(s)x_t(s)\,\mathrm ds$
and
\begin{align*}
\xi_N^i(t)
&=\operatorname{col}\left(x_t(0),x_t(-\tau_i),
\frac{\Omega_0^i(t)}{\tau_i},\ldots,
\frac{\Omega_{N-1}^i(t)}{\tau_i}\right),\\
\Gamma_N(k)
&=\begin{bmatrix}I_n&(-1)^{k+1}I_n&
\gamma_{Nk}^0I_n&\cdots&\gamma_{Nk}^{N-1}I_n\end{bmatrix},\\
\gamma_{Nk}^j
&=\begin{cases}
-(2j+1)(1-(-1)^{k+j}),&0\le j<k,\\
0,&k\le j<N.
\end{cases}
\end{align*}
For $N=0$ we have 
$\xi_0^i(t)=\operatorname{col}(x_t(0),x_t(-\tau_i))$ and
$\Gamma_0(0)=[\,I_n\ -I_n\,]$. 
\end{lemma} 
\begin{proof}
By integration by parts and the differentiation formula for the
Legendre polynomials, 
\[
\int_{-\tau_i}^0L_k^i(s){\dot x}_t(s)\,\mathrm ds
=\Gamma_N(k)\xi_N^i(t),\qquad k=0,\ldots,N.
\]
Apply Bessel's inequality in $L^2([-\tau_i,0],\mathbb R^n)$
to $R_i^{1/2}{\dot x}_t$ and the orthogonal family $L_k^i$.
The identity
$\int_{-\tau_i}^0(L_k^i)^2\,\mathrm ds=\tau_i/(2k+1)$
gives~\eqref{eq:bessel}.
\end{proof}
\begin{remark}
For $N=0$, the Bessel--Legendre inequality~\eqref{eq:bessel} is Jensen's inequality, whereas for $N=1$ 
it yields the Wirtinger-based inequality used in delay-system
analysis; see \cite{seuret2013wirtinger,Fridman2014}.
\end{remark}
\subsection{Dissipativity via nonlinear matrix inequalities}

In this subsection, we derive a delay-dependent dissipativity condition
using a Lyapunov--Krasovskii functional and the Bessel--Legendre
inequality. For $N\in\mathbb N_0$, consider
\begin{equation}\label{eq:storage_fixed_delay_nonlinear}
\begin{aligned}
\mathcal S_N(x_t)={}&\mathcal S_1(x(t))
+\hat x_N(t)^\top P_N\hat x_N(t)
+\sum_{i=1}^q\int_{t-\tau_i}^t\Theta_i(x(s))\,\mathrm ds\\
&+\sum_{i=1}^q\tau_i
\int_{t-\tau_i}^t\int_\theta^t
\dot x(s)^\top R_i\dot x(s)\,\mathrm ds\,\mathrm d\theta,
\end{aligned}
\end{equation}
where
\[
\hat x_N(t)=\operatorname{col}\bigl(\Omega_0^1(t),\ldots,
\Omega_{N-1}^1(t),\ldots,\Omega_0^q(t),\ldots,
\Omega_{N-1}^q(t)\bigr)\in\mathbb R^{qnN}.
\]
Throughout, we assume that
$\mathcal S_1:\mathbb R^n\to[0,\infty)$ is continuously
differentiable with $\mathcal S_1(0)=0$, that
$\Theta_i:\mathbb R^n\to[0,\infty)$ is continuous with
$\Theta_i(0)=0$, and that
$P_N\succeq0$ and $R_i\succeq0$ have sizes
$qnN\times qnN$ and $n\times n$, respectively.
These assumptions ensure that $\mathcal S_N$ is a nonnegative
functional on $H^1([-\tau_q,0],\mathbb R^n)$ that vanishes at zero.

Put $d_N=n(q+1+qN)$. For
\[
z=\operatorname{col}(x_0,\ldots,x_q,
\eta_{_{1,0}},\ldots,\eta_{_{1,N-1}},\ldots,
\eta_{_{q,0}},\ldots,\eta_{_{q,N-1}})\in\mathbb R^{d_N},
\]
write $\eta_{_{i,j}}=\Omega_j^i(t)/\tau_i$ when $z$ is evaluated along
a solution, and set $f_0=f(x_0,\ldots,x_q)$,
$g_0=g(x_0)$, $h_0=h(x_0)$, and
$Q_\tau=\displaystyle \sum_{i=1}^q\tau_i^2R_i$.
For $i=1,\ldots,q$ and $k=0,\ldots,N$, define the matrix
$\widetilde\Gamma_{N,i}(k)\in\mathbb R^{n\times d_N}$ by
\begin{equation*}
\widetilde\Gamma_{N,i}(k)z
=x_0-(-1)^kx_i+\sum_{j=0}^{k-1}\gamma_{Nk}^j\eta_{_{i,j}}.
\end{equation*}
With the blocks ordered first by delay and then by polynomial degree,
let
\begin{align*}
\Gamma_{N,i}
&=\operatorname{col}\bigl(\widetilde\Gamma_{N,i}(0),\ldots,
\widetilde\Gamma_{N,i}(N)\bigr),\\
\mathbf\Gamma_N
&=\operatorname{col}(\Gamma_{N,1},\ldots,\Gamma_{N,q}),\\
R_{N,i}&=\operatorname{diag}(R_i,3R_i,\ldots,(2N+1)R_i),\\
\mathbf R_N&=\operatorname{diag}(R_{N,1},\ldots,R_{N,q}),\\
G_N&=\begin{bmatrix}
0_{qnN\times n(q+1)}&
\operatorname{diag}(\tau_1I_{nN},\ldots,\tau_qI_{nN})
\end{bmatrix},\\
H_N&=\operatorname{col}\bigl(
\widetilde\Gamma_{N,1}(0),\ldots,
\widetilde\Gamma_{N,1}(N-1),\ldots,
\widetilde\Gamma_{N,q}(0),\ldots,
\widetilde\Gamma_{N,q}(N-1)\bigr).
\end{align*}
Thus $G_N,H_N\in\mathbb R^{qnN\times d_N}$,
$\mathbf\Gamma_N\in\mathbb R^{qn(N+1)\times d_N}$, and
$\mathbf R_N\in\mathbb R^{qn(N+1)\times qn(N+1)}$.
When $N=0$, the terms containing $G_N,H_N$ or $P_N$ are absent. Define
\begin{align*}
\Psi_{xx}(z)&={}\nabla\mathcal S_1(x_0)^\top f_0
+2z^\top\sym(G_N^\top P_NH_N)z+f_0^\top Q_\tau f_0\\
&+\sum_{i=1}^q\bigl(\Theta_i(x_0)-\Theta_i(x_i)\bigr)-z^\top\mathbf\Gamma_N^\top\mathbf R_N\mathbf\Gamma_Nz
-h_0^\top S_{11}h_0,\\
\Psi_{xu}(z)&={}\tfrac12\nabla\mathcal S_1(x_0)^\top g_0
+f_0^\top Q_\tau g_0-h_0^\top S_{12},\\
\Psi_{uu}(z)&={}g_0^\top Q_\tau g_0-S_{22}.
\end{align*}

The following theorem, is our main result and provides a~sufficient condition for dissipativity of delay systems based on a nonlinear matrix
inequality. It is inspired by a~general result to derive a~storage
function for nonlinear systems, see e.g.~\cite{Sch17}.

\begin{theorem}\label{thm:delay_dissipative}
For fixed $\tau_q\ge\cdots\ge\tau_1>0$ consider system~\eqref{eq:cnt_time_delay_system_nonlinear}. If a functional
$\mathcal S_N$ of the form~\eqref{eq:storage_fixed_delay_nonlinear} exists with the stated assumptions on the coefficient matrices 
and such that
\begin{equation}\label{eq:nonlinear_LMI1}
\begin{bmatrix}
\Psi_{xx}(z)&\Psi_{xu}(z)\\
\Psi_{xu}(z)^\top&\Psi_{uu}(z)
\end{bmatrix}\preceq0
\quad\text{for every }z\in\mathbb R^{d_N},
\end{equation}
then system~\eqref{eq:cnt_time_delay_system_nonlinear} is dissipative
with respect to the supply rate $s$ given by~\eqref{eq:supply_rate} on every
interval of existence.
\end{theorem}
\begin{proof}
Let $x$ be a maximal solution. 
The derivative formula for the Legendre polynomials and integration
by parts give, for $k=0,\ldots,N-1$,
\[
\dot\Omega_k^i(t)
=x(t)-(-1)^kx(t-\tau_i)
+\sum_{j=0}^{k-1}\gamma_{Nk}^j
\frac{\Omega_j^i(t)}{\tau_i}
=\widetilde\Gamma_{N,i}(k)z(t).
\]
Consequently, $\hat x_N(t)=G_Nz(t)$ and
$\dot{\hat x}_N(t)=H_Nz(t)$ for almost every $t$.
The Leibniz rule yields
\begin{align*}
\frac{\mathrm d}{\mathrm dt}
\int_{t-\tau_i}^t\Theta_i(x(s))\,\mathrm ds
&=\Theta_i(x(t))-\Theta_i(x(t-\tau_i)),\\
\frac{\mathrm d}{\mathrm dt}
\left(\tau_i\int_{t-\tau_i}^t\int_\theta^t
\dot x(s)^\top R_i\dot x(s)\,\mathrm ds\,\mathrm d\theta\right)
&=\tau_i^2\dot x(t)^\top R_i\dot x(t)
-\tau_i\int_{t-\tau_i}^t\dot x(s)^\top R_i\dot x(s)\,\mathrm ds.
\end{align*}
By Lemma~\ref{lemma:legendre}, the last integral satisfies
\[
\sum_{i=1}^q\tau_i\int_{t-\tau_i}^t
\dot x(s)^\top R_i\dot x(s)\,\mathrm ds
\ge z(t)^\top\mathbf\Gamma_N^\top\mathbf R_N
\mathbf\Gamma_Nz(t).
\]
Combining these identities, using
$\dot x(t)=f_0+g_0u(t)$ and subtracting the supply rate gives
\begin{align*}
\dot{\mathcal S}_N(x_t)-s(y(t),u(t))\le{}&
\Psi_{xx}(z(t))+2\Psi_{xu}(z(t))u(t)+u(t)^\top\Psi_{uu}(z(t))u(t)\\
={}&\begin{bmatrix}1\\u(t)\end{bmatrix}^\top
\begin{bmatrix}
\Psi_{xx}(z(t))&\Psi_{xu}(z(t))\\
\Psi_{xu}(z(t))^\top&\Psi_{uu}(z(t))
\end{bmatrix}
\begin{bmatrix}1\\u(t)\end{bmatrix}\le0.
\end{align*}
In particular, the terms linear in $u$ obtained from the
storage derivative minus the supply rate are
\[
\nabla\mathcal S_1(x_0)^\top g_0u
+2f_0^\top Q_\tau g_0u-2h_0^\top S_{12}u
=2\Psi_{xu}(z)u.
\]
Integration over $[t_0,t_1]\subset[0,T_{\max})$ proves
\eqref{eq:dissip_ineq_cont_time}.
\end{proof}

Setting $P_N=0$ and $R_i=0$ leads to the simpler storage functional
\begin{equation}\label{eq:storage_function_different_delays}
\mathcal S(x_t)=\mathcal S_1(x(t))
+\sum_{i=1}^q\int_{t-\tau_i}^t\Theta_i(x(s))\,\mathrm ds,
\end{equation}
which is used in the following corollary.
\begin{corollary}
Consider a storage functional $\mathcal{S}$ given by \eqref{eq:storage_function_different_delays} and assume that
$\mathcal S_1:\mathbb R^n\to[0,\infty)$ is continuously
differentiable with $\mathcal S_1(0)=0$, that
$\Theta_i:\mathbb R^n\to[0,\infty)$ is continuous with
$\Theta_i(0)=0$. If, for all
$x_0,\ldots,x_q\in\mathbb R^n$,
\begin{equation*}
\begin{bmatrix}
\Psi_{xx}(x_0,\ldots,x_q)&\Psi_{xu}(x_0)\\
\Psi_{xu}(x_0)^\top&\Psi_{uu}
\end{bmatrix}\preceq0,
\end{equation*}
where
\begin{align*}
\Psi_{xx}(x_0,\ldots,x_q)
={}&\nabla\mathcal S_1(x_0)^\top f(x_0,\ldots,x_q)
+\sum_{i=1}^q\bigl(\Theta_i(x_0)-\Theta_i(x_i)\bigr)\\
&-h(x_0)^\top S_{11}h(x_0),\\
\Psi_{xu}(x_0)
={}&\tfrac12\nabla\mathcal S_1(x_0)^\top g(x_0)
-h(x_0)^\top S_{12},\\
\Psi_{uu}={}&-S_{22},
\end{align*}
then system~\eqref{eq:cnt_time_delay_system_nonlinear} is dissipative
independently of the delays, with respect to the quadratic supply
rate $s$ given by \eqref{eq:supply_rate}.
\end{corollary}
\begin{proof}
For any positive ordered delays, differentiate
\eqref{eq:storage_function_different_delays} along solutions.
The resulting upper bound on $\dot{\mathcal S}-s$ is the quadratic
form in the stated matrix. Since the matrix condition is independent
of the delay values, it applies to every such choice.
\end{proof}
\subsection{Interconnection of passive time-delay systems}
In the special case of impedance passive systems, the passivity is preserved under a particular type of negative feedback interconnections. For a similar result for nonlinear delay-systems using a slightly weaker notion of passivity we refer to \cite{kawano2023passivity}. 

We consider two impedance passive time-delay systems of the form
\begin{align}\label{eq:Interconnection}
\nonumber 
\dot x_j(t) &= f_j(x_j(t),x_j(t-\tau_{1,j}),\ldots, x_j(t-\tau_{q,j})) 
  + g_j(x_j(t))u_j(t), \quad t\geq 0,\\
y_j(t) &= h_j(x_j(t)),\quad \quad j=1,2,
\end{align}
with storage functions $\mathcal{S}_1$ and $\mathcal{S}_2$, respectively.
\begin{proposition}
If the systems \eqref{eq:Interconnection} are impedance passive with storage functions $\mathcal{S}_1$ and $\mathcal{S}_2$ and compatible input--output dimensions, then, whenever the resulting closed-loop system has a solution, the negative feedback interconnection
\[
u_1=v-y_2, \quad u_2=y_1,
\]
with external input $v$ is impedance passive with storage function $\mathcal{S}=\mathcal{S}_1+\mathcal{S}_2$ on every common interval of existence.  
\end{proposition}
\begin{proof}
The impedance passivity inequalities hold individually for arbitrary inputs which gives
\begin{align*}
\mathcal{S}(x_{t_1})-\mathcal{S}(x_{t_0})&=\mathcal{S}_1(x_{t_1}^1)-\mathcal{S}_1(x_{t_0}^1)+\mathcal{S}_2(x_{t_1}^2)-\mathcal{S}_2(x_{t_0}^2) \\ 
&\leq \int_{t_0}^{t_1}y_1(s)^\top u_1(s)\mathrm{d}s+\int_{t_0}^{t_1}y_2(s)^\top u_2(s)\mathrm{d}s\\
&=\int_{t_0}^{t_1}y_1(s)^\top (v-y_2)(s)\mathrm{d}s+\int_{t_0}^{t_1}y_2(s)^\top y_1(s)\mathrm{d}s\\
&=\int_{t_0}^{t_1}y_1(s)^\top v(s)\mathrm{d}s.
\end{align*}
Hence, the interconnected system is impedance passive with respect to the input $v$ and the output $y_1$.
\end{proof}

\section{Dissipativity for linear time-delay systems}
\label{sec:linear}

In this section, we derive an LMI condition for linear systems with
multiple delays and a general quadratic supply rate.
Consider the system
\begin{align}
\dot x(t)&=Ax(t)+\sum_{i=1}^qA_ix(t-\tau_i)+Bu(t),\quad 
x(t)=\phi(t),\quad t\in[-\tau_q,0],
\label{eq:cnt_time_delay_system}\\
y(t)&=Cx(t),\nonumber
\end{align}
where $A,A_i\in\mathbb R^{n\times n}$,
$B\in\mathbb R^{n\times m}$, $C\in\mathbb R^{m\times n}$,
and $\tau_q\ge\cdots\ge\tau_1>0$.

The following condition is based on the Bessel--Legendre estimate and
the differentiation formulas of the preceding section. We allow cross
terms between the current state and the Legendre moments in the
quadratic storage functional
\begin{equation}\label{eq:storage_fixed_delay_linear}
\begin{aligned}
\mathcal S_N(x_t)={}&\widetilde x_N(t)^\top\widetilde P_N
\widetilde x_N(t)
+\sum_{i=1}^q\int_{t-\tau_i}^t
x(s)^\top\Theta_i x(s)\,\mathrm ds\\
&+\sum_{i=1}^q\tau_i\int_{t-\tau_i}^t\int_\theta^t
\dot x(s)^\top R_i\dot x(s)\,\mathrm ds\,\mathrm d\theta,
\end{aligned}
\end{equation}
where $\widetilde x_N(t)=\operatorname{col}(x(t),\hat x_N(t))
\in\mathbb R^{n+qnN}$ and, for $N=0$,
$\widetilde x_0(t)=x(t)$. Here
$\widetilde P_N\in\mathbb R^{(n+qnN)\times(n+qnN)}$ and
$\Theta_i,R_i\in\mathbb R^{n\times n}$ are symmetric positive
semidefinite matrices.
For $d_N=n(q+1+qN)$, let
$\widetilde z=\operatorname{col}(z,u)\in\mathbb R^{d_N+m}$, where
$z$ is ordered as in the nonlinear section, and define
\begin{align*}
F_N&=\begin{bmatrix}
A&A_1&\cdots&A_q&0_{n\times qnN}&B
\end{bmatrix},\\
\widetilde G_N&=\begin{bmatrix}
I_n&0_{n\times qn}&0_{n\times qnN}&0_{n\times m}\\
0_{qnN\times n}&0_{qnN\times qn}&
\operatorname{diag}(\tau_1I_{nN},\ldots,\tau_qI_{nN})&
0_{qnN\times m}
\end{bmatrix},\\
\widetilde H_N&=\operatorname{col}\bigl(
F_N,\,[\,H_N\ \ 0_{qnN\times m}\,]\bigr),\\
\overline\Gamma_N&=\begin{bmatrix}
\mathbf\Gamma_N&0_{qn(N+1)\times m}
\end{bmatrix},\\
M_{1,\tau}&=F_N^\top Q_\tau F_N,
\qquad Q_\tau=\sum_{i=1}^q\tau_i^2R_i.
\end{align*}
Thus $F_N\in\mathbb R^{n\times(d_N+m)}$,
$\widetilde G_N,\widetilde H_N\in
\mathbb R^{(n+qnN)\times(d_N+m)}$, and
$\overline\Gamma_N\in\mathbb R^{qn(N+1)\times(d_N+m)}$.
Furthermore, 
set
\begin{equation*}
M_0=\begin{bmatrix}
\sum_{i=1}^q\Theta_i-C^\top S_{11}C
&0&0&-C^\top S_{12}\\
0&-\operatorname{diag}(\Theta_1,\ldots,\Theta_q)&0&0\\
0&0&0&0\\
-S_{12}^\top C&0&0&-S_{22}
\end{bmatrix}.
\end{equation*}

\begin{proposition}
\label{cor:linear-LMI}
For fixed $\tau_q\ge\cdots\ge\tau_1>0$, the linear delay system
\eqref{eq:cnt_time_delay_system} is dissipative with respect to
\eqref{eq:supply_rate} if there exist symmetric positive semidefinite
matrices $\Theta_i,R_i\in\mathbb R^{n\times n}$ and
$\widetilde P_N\in\mathbb R^{(n+qnN)\times(n+qnN)}$
such that
\begin{equation}\label{eq:LMI_fixed_tau}
2\sym(\widetilde G_N^\top\widetilde P_N\widetilde H_N)
+M_0+M_{1,\tau}
-\overline\Gamma_N^\top\mathbf R_N\overline\Gamma_N
\preceq0.
\end{equation}
\end{proposition}
\begin{proof}
Let $\widetilde z(t)=\operatorname{col}(z(t),u(t))$.
By construction,
$\widetilde x_N(t)=\widetilde G_N\widetilde z(t)$,
$\dot{\widetilde x}_N(t)=\widetilde H_N\widetilde z(t)$,
and $\dot x(t)=F_N\widetilde z(t)$.
Differentiating the storage functional and applying
Lemma~\ref{lemma:legendre} yields, for almost every $t$,
\begin{align*}
\dot{\mathcal S}_N(x_t)-s(y(t),u(t))
\le{}&\widetilde z(t)^\top
\bigl(2\sym(\widetilde G_N^\top\widetilde P_N\widetilde H_N)
+M_0+M_{1,\tau}
-\overline\Gamma_N^\top\mathbf R_N\overline\Gamma_N\bigr)
\widetilde z(t).
\end{align*}
The inequality~\eqref{eq:LMI_fixed_tau} therefore implies
$\dot{\mathcal S}_N(x_t)\le s(y(t),u(t))$ almost everywhere.
Integrating proves~\eqref{eq:dissip_ineq_cont_time}.
\end{proof}

\begin{remark}
For impedance passivity, $S_{22}=0$. The input--input block
of~\eqref{eq:LMI_fixed_tau} then gives
$B^\top Q_\tau B\preceq0$. Since
$Q_\tau=\displaystyle\sum_{i=1}^q\tau_i^2R_i\succeq0$ and $\tau_i>0$,
we obtain $R_iB=0$ for every $i$. Equivalently,
$\ran B\subseteq\bigcap_{i=1}^q\ker R_i$.
In particular, if $B\ne0$, none of the $R_i$ can be positive definite.
\end{remark}

\section{Examples}\label{sec:examples}
We give two numerical examples illustrating that higher-order Bessel--Legendre
estimates can yield strictly stronger sufficient dissipativity conditions.
The first uses the nonlinear condition without modifying its storage
functional, while the second exploits the cross terms allowed in the
linear LMI.

\subsection{A scalar nonlinear system}
We consider an example where the Bessel--Legendre inequality can guarantee the scattering passivity for $N=2$ but not for $N=1$. To this end, we consider
\begin{equation}\label{eq:ex-scalar}
\begin{aligned}
 \dot x(t)&=-d(x(t))x(t)-x(t-\tau)+\tfrac1{10}u(t),\\
 y(t)&=\tfrac1{10}x(t),\qquad
 d(x)=\tfrac35+\frac{x^2}{20(1+x^2)},\qquad \tau=\tfrac52.
\end{aligned}
\end{equation}
The undelayed system is port-Hamiltonian with
$\mathcal H(x)=x^2/2$, $J=0$, \\ $R(x)=d(x)>0$, and $B=\tfrac{1}{10}$ and 
the additional term $-x(t-\tau)$ models delayed feedback.
We consider the scattering supply rate
$s_{sca}(y,u)=\tfrac12(u^2-y^2)$.
In~\eqref{eq:storage_fixed_delay_nonlinear} choose $N=2$,
\[
 \mathcal S_1(x)=x^2,\quad \Theta_1(x)=\tfrac45x^2,\quad
 R_1=\tfrac{17}{100},\quad
 P_2=\frac{2}{5\tau^2}
\begin{bmatrix}1&-\tfrac74\\-\tfrac74&\tfrac{49}{16}\end{bmatrix}\succeq0.
\]
With $\eta_j=\tfrac{\Omega_j^1}{\tau}$ for $j=0,1$, the storage reads
\begin{equation}\label{eq:ex-scalar-storage}
\begin{aligned}
\mathcal S_2(x_t)={}&x(t)^2
 +\tfrac25\bigl(\eta_0(t)-\tfrac74\eta_1(t)\bigr)^2
 +\tfrac45\int_{t-\tau}^t x(s)^2\,\mathrm ds
+\tfrac{17\tau}{100}\int_{t-\tau}^t\int_\theta^t
 \dot x(s)^2\,\mathrm ds\,\mathrm d\theta.
\end{aligned}
\end{equation}
To verify the condition in Theorem~\ref{thm:delay_dissipative}, write 
$\alpha=d(x(t))\in[\tfrac35,\tfrac{13}{20})$,\\
$F_\alpha(t)=-\alpha x(t)-x(t-\tau)+\tfrac{1}{10}u(t)$, and
$w=\eta_0-\tfrac{7}{4}\eta_1$. Differentiation and the 
Bessel--Legendre inequality for $N=2$ yield
\begin{equation}\label{eq:ex-scalar-phi}
\begin{aligned}
\dot{\mathcal S}_2-s_{sca}(y,u)&\le{}2x(t)F_\alpha(t)+\frac{4w}{5\tau}
 \left(x(t)-x(t-\tau)-\tfrac74(x(t)+x(t-\tau)-2\eta_0)\right)
\\&+\tfrac45(x(t)^2-x(t-\tau)^2)
+\tfrac{17}{100}\bigl(\tau^2F_\alpha(t)^2-(x(t)-x(t-\tau))^2\\
&-3(x(t)+x(t-\tau)-2\eta_0)^2
      -5(x(t)-x(t-\tau)-6\eta_1)^2\bigr)\\
      &+\tfrac1{200}x(t)^2-\tfrac12u(t)^2\\
      &=z(t)^\top M(\alpha(t))z(t)
\end{aligned}
\end{equation}
for $z(t)=(x(t),x(t-\tau),\eta_0(t),\eta_1(t),u(t))^\top$ and some symmetric matrix $M(\alpha)\in\mathbb R^{5\times5}$. 
Direct calculation shows that for
$\alpha\in\{3/5,13/20\}$ all five leading principal minors of
$-M(\alpha)-I_5/250$ are positive 
and $M(3/5),M(13/20)\prec-I_5/250$. 
\[
 M(\alpha)={}(1-\lambda)M(3/5)+\lambda M(13/20)
 -\frac{17\tau^2}{100}\lambda(1-\lambda)
    (13/20-3/5)^2e_1e_1^\top.
\]
Therefore $M(d(a))\preceq-I_5/250$ for every $a$, and
\eqref{eq:ex-scalar} is scattering passive by
Theorem~\ref{thm:delay_dissipative}.

We next show that the sufficient condition in
Theorem~\ref{thm:delay_dissipative} is not feasible
for $N=1$ and $\tau=5/2$.
Suppose, to the contrary, that there exist admissible
functions $\mathcal S_1$ and $\Theta_1$ and constants
$P_1=p\ge0$ and $R_1=r\ge0$ satisfying
\eqref{eq:nonlinear_LMI1}.

We first write this condition explicitly for
system~\eqref{eq:ex-scalar}.
Let $a,b,\eta\in\mathbb R$ and $u\in\mathbb R$, and set
\[
v=-d(a)a-b+\frac{u}{10}.
\]
Here $a$, $b$, and $\eta$ represent the current state,
the delayed state, and the normalized Legendre moment
$\Omega_0^1/\tau$, respectively.
By~\eqref{eq:nonlinear_LMI1}, the quadratic form
associated with its matrix, evaluated at
$\operatorname{col}(1,u)$, must be nonpositive.
Using~\eqref{eq:storage_fixed_delay_nonlinear}
and the Bessel--Legendre inequality~\eqref{eq:bessel}
with $N=1$, this gives
\begin{align}\nonumber
0&\ge{}
\mathcal S_1'(a)v
+2\tau p\eta(a-b)
+\Theta_1(a)-\Theta_1(b)\\
&+r\Bigl(
\tau^2v^2-(a-b)^2
-3(a+b-2\eta)^2
\Bigr)
+\frac{a^2}{200}-\frac{u^2}{2}.
\label{eq:ex-scalar-N1-ineq}
\end{align}
In particular, the term $2\tau p\eta(a-b)$
follows from differentiating
$p(\Omega_0^1)^2$ and using
$\dot\Omega_0^1=a-b$.
The last two terms arise from the scattering
supply rate $s(y,u)=(u^2-y^2)/2$ with $y=a/10$.

Importantly, inequality~\eqref{eq:ex-scalar-N1-ineq}
must hold for all $a,b,\eta,u$, independently
of whether these values arise from a solution
of the delay equation. We exploit this necessary
condition by choosing the following periodic functions:
\begin{equation}\label{eq:ex-scalar-N1-test}
\begin{aligned}
a(t)&=\rho\cos(\omega t),\quad 
b(t)=\rho\left(
-\frac35\cos(\omega t)
+\frac45\sin(\omega t)\right),\qquad
\rho=\frac1{10},\quad\omega=\frac45,\\
\eta(t)&=\frac{a(t)+b(t)}2,\quad
u(t)=\frac{a(t)^3}{2(1+a(t)^2)}.
\end{aligned}
\end{equation}
The input in~\eqref{eq:ex-scalar-N1-test}
is chosen such that the nonlinear part of
$d(a)a$ cancels. Indeed, by~\eqref{eq:ex-scalar},
\[
\begin{aligned}
v(t)
&=-\left(\frac35+
\frac{a(t)^2}{20(1+a(t)^2)}\right)a(t)
-b(t)+\frac{u(t)}{10}\\
&=-\frac35a(t)-b(t)\\
&=-\rho\omega\sin(\omega t)\\
&=\dot a(t).
\end{aligned}
\]
We now integrate~\eqref{eq:ex-scalar-N1-ineq}
over one period $T=2\pi/\omega$. First, since $v=\dot a$ and $a(T)=a(0)$,
\[
\int_0^T\mathcal S_1'(a(t))v(t)\,\mathrm dt
=\mathcal S_1(a(T))-\mathcal S_1(a(0))=0.
\]
Moreover, $a$ and $b$ are sinusoids with the
same amplitude and frequency, differing only
by a phase shift. Hence, for every continuous
function $\Theta_1$,
\[
\int_0^T\Theta_1(a(t))\,\mathrm dt
=\int_0^T\Theta_1(b(t))\,\mathrm dt.
\]
The term involving $P_1=p$ also integrates to
zero. Indeed, by the choice of $\eta$,
$2\tau p\eta(a-b)
=\tau p(a^2-b^2)$, and the integrals of $a^2$ and $b^2$ over
one period are equal. Next, the choice $\eta=(a+b)/2$ implies
$a+b-2\eta=0$. Thus, the additional square in the Bessel--Legendre estimate for $N=1$ 
vanishes. For the remaining terms, direct integration gives
\[
\int_0^T\dot a(t)^2\,\mathrm dt
=\frac{16}{25}\int_0^T a(t)^2\,\mathrm dt,\quad 
\int_0^T(a(t)-b(t))^2\,\mathrm dt
=\frac{16}{5}\int_0^T a(t)^2\,\mathrm dt.
\]
Since $\tau=5/2$, it follows that
\begin{equation}\label{eq:ex-scalar-N1-integral}
\int_0^T
\Bigl(\tau^2v(t)^2-(a(t)-b(t))^2\Bigr)
\,\mathrm dt
=\left(\frac{25}{4}\frac{16}{25}
-\frac{16}{5}\right)
\int_0^T a(t)^2\,\mathrm dt=\frac45\int_0^T a(t)^2\,\mathrm dt.
\end{equation}

Finally, $|a(t)|\le\rho=1/10$ implies
\[
u(t)^2
=\frac{a(t)^6}{4(1+a(t)^2)^2}
\le\frac{\rho^4}{4}a(t)^2
=\frac1{40000}a(t)^2.
\]
Consequently,
\begin{equation}\label{eq:ex-scalar-N1-supply}
\frac{a(t)^2}{200}-\frac{u(t)^2}{2}
\ge\frac{399}{80000}a(t)^2.
\end{equation}

Combining these observations with
\eqref{eq:ex-scalar-N1-ineq} and using
\eqref{eq:ex-scalar-N1-integral} and
\eqref{eq:ex-scalar-N1-supply}, we obtain
\[
0\ge
\left(\frac45r+\frac{399}{80000}\right)
\int_0^T a(t)^2\,\mathrm dt.
\]
The right-hand side is strictly positive
because $r\ge0$ and $a$ is not identically zero.
This is a contradiction.

Therefore, the sufficient condition
\eqref{eq:nonlinear_LMI1} is infeasible for
$N=1$ at $\tau=5/2$, regardless of the choice
of admissible $\mathcal S_1$, $\Theta_1$, $P_1$,
and $R_1$. 

\subsection{A linear system}
Here we show that impedance passivity can be guaranteed using the condition from Proposition~\ref{cor:linear-LMI} for $N=1$ but not for $N=0$. Consider the two-state port-Hamiltonian delay system
\begin{equation}\label{eq:ex-linear}
\begin{aligned}
 \dot x_1(t)&=x_2(t)-x_1(t-\tau),\\
 \dot x_2(t)&=-x_1(t)-10x_2(t)+u(t),\\
 y(t)&=x_2(t),\qquad \tau=\tfrac32.
\end{aligned}
\end{equation}
Its undelayed dynamics has Hamiltonian
$\mathcal H(x)=\tfrac12\|x\|^2$, interconnection matrix
$J=\left[\begin{smallmatrix}0&1\\-1&0\end{smallmatrix}\right]$,
dissipation matrix $\operatorname{diag}(0,10)$, and input matrix
$B=[0,1]^\top$.
We consider the impedance supply rate, $s_{imp}(y,u)=yu$.

For $N=1$, put
$\omega(t)=\int_{t-\tau}^t x_1(s)\,\mathrm ds$ and choose
\[
\widetilde P_1=\frac1{64}
\begin{bmatrix}
23&0&-10&0\\0&32&0&0\\-10&0&5&0\\0&0&0&0
\end{bmatrix}\succeq0,
\quad
\Theta_1=\begin{bmatrix}29/128&0\\0&0\end{bmatrix},
\quad
R_1=\begin{bmatrix}1/16&0\\0&0\end{bmatrix}.
\]
Writing $\Omega_0^1(t)=(\omega(t),\sigma(t))^\top$,
where $\sigma(t)=\int_{t-\tau}^t x_2(s)\,\mathrm ds$,
and taking $\widetilde x_1=(x_1,x_2,\omega,\sigma)^\top$,
these are precisely the matrices in
\eqref{eq:storage_fixed_delay_linear}. The storage takes the form
\begin{align}
 \mathcal S_1(x_t)&={}\frac{x_2(t)^2}{2}
 +\frac{3x_1(t)^2+5(\omega(t)-2x_1(t))^2}{64}
 +\frac{29}{128}\int_{t-\tau}^t x_1(s)^2\,\mathrm ds\\ \nonumber 
 &+\frac\tau{16}\int_{t-\tau}^t\int_\theta^t
 \dot x_1(s)^2\,\mathrm ds\,\mathrm d\theta.
\end{align}
Using the Bessel--Legendre inequality with $N=1$ gives
\[
 \dot{\mathcal S}_1(x_t)-u(t)y(t)
 \le-\frac1{128}\zeta(t)^\top K\zeta(t),\quad  K=\begin{bmatrix}
 43&42&-63&18\\42&43&-63&18\\
 -63&-63&96&30\\18&18&30&1262
 \end{bmatrix}
\]
where
$\zeta=(x_1(t),x_1(t-\tau),\omega(t)/\tau,x_2(t))^\top$. 
The leading principal minors of $K$ are all positive. Hence, the LMI~\eqref{eq:LMI_fixed_tau} for $N=1$  holds, and the system~\eqref{eq:ex-linear} is impedance passive.

To show that Proposition~\ref{cor:linear-LMI} with $N=0$ does not certify the passivity for 
$\tau=\tfrac32$, suppose that it is feasible with
$\widetilde P_0\succeq0$, $\Theta_1\succeq0$, and $R_1\succeq0$.
The input blocks imply
$R_1B=0$ and $\widetilde P_0B=\tfrac12C^\top$.
Consequently,
\[
 R_1=\operatorname{diag}(r,0),\quad
 \widetilde P_0=\operatorname{diag}(p,\tfrac12),\qquad r,p\ge0.
\]

We finally show that the condition~\eqref{eq:LMI_fixed_tau}
is not feasible for $N=0$ and $\tau=\tfrac32$.
Suppose, to the contrary, that there exist
$\widetilde P_0,\Theta_1,R_1\succeq0$
satisfying~\eqref{eq:LMI_fixed_tau}. 
Differentiating the storage functional
\eqref{eq:storage_fixed_delay_linear} and using
\eqref{eq:bessel} with $N=0$ gives
\begin{align}
\dot{\mathcal S}_0(x_t)-y(t)u(t)
&\le{}
2x(t)^\top\widetilde P_0\dot x(t)
+x(t)^\top\Theta_1x(t)
-x(t-\tau)^\top\Theta_1x(t-\tau) \nonumber\\
&-\bigl(x(t)-x(t-\tau)\bigr)^\top
R_1\bigl(x(t)-x(t-\tau)\bigr)\nonumber\\
&+\tau^2\dot x(t)^\top R_1\dot x(t)-y(t)u(t).
\label{eq:example_N0_bound}
\end{align}
The LMI~\eqref{eq:LMI_fixed_tau} requires the
right-hand side to be nonpositive for all
$x(t),x(t-\tau)\in\mathbb R^2$ and $u(t)\in\mathbb R$.

We first consider the terms involving the input.
For system~\eqref{eq:ex-linear}, we have
$B=C^\top=[0,1]^\top$ and $y=Cx$.
Since $\dot x$ contains the term $Bu$, the
coefficient of $u^2$ in~\eqref{eq:example_N0_bound}
is $\tau^2B^\top R_1B$.
As $R_1\succeq0$, this coefficient is nonnegative.
For the inequality to hold for every $u(t)$, it
must therefore vanish, which implies $R_1B=0$.
The remaining term involving $u(t)$ is
$\bigl(2x(t)^\top\widetilde P_0B-Cx(t)\bigr)u(t)$. 
It must also vanish for every $x(t)$ and $u(t)$.
Consequently, 
$R_1B=0$ and 
$\widetilde P_0B=\frac12C^\top$.
Using the symmetry and positive semidefiniteness
of both matrices, we obtain
\[
R_1=\begin{bmatrix}r&0\\0&0\end{bmatrix},
\qquad
\widetilde P_0=
\begin{bmatrix}p&0\\0&1/2\end{bmatrix},
\qquad p,r\ge0.
\]
Next, let $\theta=(\Theta_1)_{11}\ge0$ and choose
$x(t)=(x_1,0)^\top$,
$x(t-\tau)=(\hat x_1,0)^\top$, and $u=0$.
By~\eqref{eq:ex-linear},
$\dot x(t)=(-\hat x_1,-x_1)^\top$.
Substituting into~\eqref{eq:example_N0_bound}
gives
\[
-2px_1\hat x_1+\theta(x_1^2-\hat x_1^2)
+\tau^2r\hat x_1^2-r(x_1-\hat x_1)^2\le0,
\]
or, equivalently,
\begin{equation}\label{eq:example_N0_restriction}
(\theta-r)x_1^2+2(r-p)x_1\hat x_1
+\bigl((\tau^2-1)r-\theta\bigr)\hat x_1^2
\le0.
\end{equation}
Since~\eqref{eq:example_N0_restriction}
must hold for all $x_1,\hat x_1\in\mathbb R$, choosing
first $\hat x_1=0$ and then $x_1=0$ yields
$(\tau^2-1)r\le\theta\le r$. 
For $\tau=3/2$, this becomes
$\frac54r\le\theta\le r$.
As $r\ge0$, we necessarily have $r=\theta=0$.
Equation~\eqref{eq:example_N0_restriction}
then reduces to $-2px_1\hat x_1\le0$ for all $x_1,\hat x_1$,
which is possible only if $p=0$.
Furthermore, $\Theta_1\succeq0$ and
$(\Theta_1)_{11}=0$ imply
$\Theta_1=\left[\begin{smallmatrix}0&0\\0&\kappa\end{smallmatrix}\right]$, for some $\kappa\ge0$.
\medskip \\
It remains to show that these necessary conditions
contradict~\eqref{eq:LMI_fixed_tau}.
Choose
$x(t)=(x_1,x_2)^\top$, 
$x(t-\tau)=0$, and $u=0$. Then~\eqref{eq:ex-linear} gives
$\dot x(t)=(x_2,-x_1-10x_2)^\top$.
Since $R_1=0$ and
$\widetilde P_0=\operatorname{diag}(0,\tfrac12)$,
the right-hand side of~\eqref{eq:example_N0_bound}
reduces to
\[
x_2(-x_1-10x_2)+\kappa x_2^2
=-x_1x_2+(\kappa-10)x_2^2.
\]
Hence, \eqref{eq:LMI_fixed_tau} would require
$-x_1x_2+(\kappa-10)x_2^2\le0$ for all $x_1,x_2\in\mathbb R$. 
Taking $x_2=1$ and letting $x_1$ tend to
$-\infty$ leads to a contradiction. Therefore, \eqref{eq:LMI_fixed_tau} is infeasible
for $N=0$ at $\tau=3/2$.

\color{black}

\section{Conclusion}
\label{sec:conclusion}
In this paper, we studied dissipativity properties for a class of nonlinear time-delay systems with multiple constant delays. By employing Lyapunov--Krasovskii storage functionals together with Bessel--Legendre inequalities, we derived sufficient delay-dependent dissipativity conditions with respect to general quadratic supply rates. A delay-independent criterion was also obtained as a special case by considering a simplified storage functional.
\medskip\\
For linear time-delay systems, the proposed framework led to tractable sufficient conditions expressed in terms of linear matrix inequalities. These conditions can be checked numerically using standard semidefinite programming tools and provide a systematic way to verify dissipativity, including passivity-type properties, in the presence of multiple delays.
\vspace{2mm}\\
Future work will focus on reducing the conservatism of the obtained conditions and developing numerical methods for the nonlinear case. Although the proposed matrix inequality is affine in the unknown storage functions and matrices, it must hold for all admissible state and moment variables. Suitable finite-dimensional parametrizations and methods for verifying these inequalities are therefore required. Another direction is the application of the criteria to physical systems with delayed feedback or delayed damping.

\section*{Acknowledgments}
This work was funded by the Deutsche Forschungsgemeinschaft (DFG, German
Research Foundation) – Project-ID 531152215 – CRC~1701. Ikram El Haskouki gratefully acknowledges University of Wuppertal, Germany, for hosting her during her research stay and for their support.

\subsection*{AI disclosure}
The generative AI model Chat-GPT 6 Astra was used in the creation and verification of the two numerical examples that are presented in Section 4. All mathematical ideas, methods, 
arguments, and proofs are the authors' own. Every text generated by AI was critically reviewed by the authors who take responsibility for the content of this manuscript.

\color{black}

\vspace{\baselineskip}

\bibliographystyle{abbrv}
\bibliography{references_P.bib}

\end{document}